\documentclass[11pt,twoside,letter]{article}
\usepackage{amsthm, amsmath, amssymb, amsfonts, graphicx, epsfig}
\usepackage{bbm}
\usepackage{mathrsfs}
\usepackage{dsfont}
\usepackage{epstopdf}
\usepackage{enumerate}
\usepackage[usenames,dvipsnames]{color}

\usepackage[colorlinks=true, pdfstartview=FitV, linkcolor=blue,
            citecolor=blue, urlcolor=blue]{hyperref}
\usepackage[usenames]{color}

\usepackage{url}

\makeatletter\def\url@leostyle{%
 \@ifundefined{selectfont}{\def\UrlFont{\sf}}{\def\UrlFont{\scriptsize\ttfamily}}} \makeatother

\usepackage{array}
\usepackage[margin=1.0in, letterpaper]{geometry}
\usepackage{xcolor} 

\newtheorem{theorem}{Theorem}
\newtheorem{proposition}[theorem]{Proposition}
\newtheorem{lemma}[theorem]{Lemma}

\theoremstyle{definition}

\theoremstyle{remark}
\newtheorem{remark}[theorem]{Remark}

\numberwithin{equation}{section}
\numberwithin{theorem}{section}

\definecolor{Red}{rgb}{0.9,0,0.0}
\definecolor{Blue}{rgb}{0,0.0,1.0}

\def\cH{\mathcal{H}}
\def\cI{\mathcal{I}}

\def\cN{\mathcal{N}}

\def\bE{\mathbb{E}}

\def\bI{\mathbb{I}}

\def\bN{\mathbb{N}}

\def\bP{\mathbb{P}}

\def\bR{\mathbb{R}}

\def\bV{\mathbb{V}}

\def\sF{\mathscr{F}}

\newcommand{\wh}{\widehat}

\newcommand{\1}{\mathbbm{1}}            % preferable way of writing indicator function
\newcommand{\set}[1]{\{#1\}}            % set: {xyz} to be used for inline formulas
\newcommand{\Laplace}{\boldsymbol{\Delta}}       % BS\DeltaE

\DeclareMathOperator*{\w}{w-\!\!} % limit in distribution
\DeclareMathOperator{\dif}{d \!}        % used for differential, same as in commath.sty
\DeclareMathOperator{\Var}{\bV\mathrm{ar}}          % \Var for variance
\newcommand{\ch}{\mathcal{H}}
\newcommand{\RR}{\mathbb R}

\newcommand{\E}{\mathbb E}

\newcommand{\cf}{\mathcal F}

\newcommand{\cn}{\mathcal N}

\newcommand{\ga}{\gamma}

\newcommand{\si}{\sigma}

\title{Drift estimation for discretely sampled SPDEs}

\date{ {\small  %\today\\ Preliminary Draft}} %
First Circulated: \today }}
\begin{document}
\author{
    Igor Cialenco\,\thanks{Department of Applied Mathematics, Illinois Institute of Technology
       \newline \hspace*{1.45em}  10 W 32nd Str, Building REC, Room 208, Chicago, IL 60616, USA
       \newline \hspace*{1.45em}  Email: \url{cialenco@iit.edu},  URL: \url{http://math.iit.edu/\~igor}
        \vspace{0.5em}}
\and
 Francisco Delgado-Vences\,\thanks{ Instituto de Matem\'aticas, UNAM, Oaxaca de Juarez, M\'exico 61000
         \newline \hspace*{1.45em}  Email: \url{delgado@im.unam.mx}
         \vspace{0.5em}}
\and
        Hyun-Jung Kim\,\thanks{ Department of Applied Mathematics, Illinois Institute of Technology
       \newline \hspace*{1.45em}  10 W 32nd Str, Building REC, Room 208, Chicago, IL 60616, USA
         \newline \hspace*{1.45em} Email: \url{hkim129@iit.edu}, URL: \url{https://sites.google.com/view/hyun-jungkim}
         }
        % if the same address, then use as in previous example
        }

\maketitle

\vspace{-2em}

\smallskip

{\footnotesize
\begin{tabular}{l@{} p{350pt}}
  \hline \\[-.2em]
  \textsc{Abstract}: \ &  The aim of this paper is to study the asymptotic properties of the maximum likelihood estimator (MLE) of the drift coefficient for fractional stochastic heat equation  driven by an additive space-time noise. We consider the traditional for stochastic partial differential equations  statistical experiment when the measurements are performed in the spectral domain, and in contrast to the existing literature, we study the asymptotic properties of the maximum likelihood (type) estimators (MLE) when both, the number of Fourier modes and the time go to infinity. In the first part of the paper we consider the usual setup of continuous time observations of the Fourier coefficients of the solutions, and show that the MLE is consistent, asymptotically normal and optimal in the mean-square sense. In the second part of the paper  we investigate  the natural time discretization of the MLE, by assuming that the first $N$ Fourier modes are measured at $M$ time grid points, uniformly spaced over the time interval $[0,T]$. We provide a rigorous asymptotic analysis of the proposed estimators when $N\to\infty$ and/or $T,M\to\infty$. We establish sufficient conditions  on the growth rates of $N,M$ and $T$, that guarantee consistency and asymptotic normality of these estimators.

 \\[0.5em]
\textsc{Keywords:} \ & fractional stochastic heat equation, parabolic SPDE, stochastic evolution equations, statistical inference for SPDEs, drift estimation, discrete sampling, high-frequency sampling.   \\
\textsc{MSC2010:} \ & 60H15, 65L09, 62M99 \\[1em]
  \hline
\end{tabular}
}

\section{Introduction} Undoubtedly, the stochastic partial differential equations (SPDEs) serve as a modern powerful modeling tool in describing the evolution of dynamical systems in the presence of spatial-temporal uncertainties with particular applications in fluid mechanics, oceanography, temperature anomalies, finance, economics, biological and ecological systems, and many other applied disciplines. Major breakthrough results have been established on the general analytical theory for SPDEs, such as existence, uniqueness and regularity properties of the solutions.  For an in depth discussion of the theory of SPDEs and their various applications, we refer to recent monographs~\cite{LototskyRozovsky2017Book,RozovskyRozovsky2018Book}. In contrast, the investigation of inverse problems for SPDEs, and in particular parameter estimation problems, are  still in their emerging phase. We refer to the survey papers \cite{Lototsky2009Survey,Cialenco2018} and the monograph \cite[Chapter~6]{LototskyRozovsky2017Book} for an overview of the literature and existing methodologies on statistical inference for parabolic SPDEs. Most of the existing results are obtained within the so-called spectral approach, when it is assumed that the observer measures the values of one realization of the first $N$ Fourier modes of the solution continuously over a finite time interval $[0,T]$. In such cases, usually the statistical problems are addressed via maximum likelihood estimators (MLEs), and the asymptotic properties of the estimators are studied in the large number of Fourier modes regime, $N\rightarrow\infty$, while time horizon $T$ is fixed. The large time asymptotics regime $T\to\infty$, while $N$ being fixed, usually falls in the realm of finite dimensional stochastic differential equations, which is a well-established research area. This asymptotic regime in the context of SPDEs was only briefly discussed in~\cite{LototskyRozovsky2017Book,CXu2015}. Only several works have been dedicated to parameter estimation problems for SPDEs in discrete  sampling setup.  In~\cite{PiterbargRozovskii1997,PrakasaRao2002,PrakasaRao2003,Markussen2003}, the authors investigate some version of the discretized MLEs for some particular equations. In  \cite{PospivsilTribe2007}, and more recently in \cite{CialencoHuang2017,BibingerTrabs2017,Chong2019,BibingerTrabs2019}, using various approaches the authors study the estimation of the drift and/or volatility coefficients when the solution is sampled discretely in physical domain.

\smallskip\noindent
\textbf{The aim of this work} is to provide a rigorous and comprehensive asymptotic analysis of the time discretized MLE for the drift coefficient of a fractional heat equation driven by an additive space-time noise. The precise form of the considered equations and their well posedness are presented in Section~\ref{sec:Setup}. The main results of this work can be summarized as follows:
 \begin{enumerate}[$\rhd$]
   \item In Section~\ref{sec:contTime} we assume the same sampling scheme as in the existing literature on spectral approach, namely continuous time observations of the $N$ Fourier modes for $t\in[0,T]$, and study the asymptotic properties of the MLE, when \textit{both}, $N,T\to\infty$. We prove that this estimator is (strongly) consistent, and asymptotically normal. We give two proofs of the asymptotic normality, one based on Malliavin calculus, which we believe can be used, with slight modifications to other similar problems. Another proof uses classical results from general probability and explodes the particular structure and properties of the underlying problem. In particular, we show that the estimator is optimal in the mean-square sense.
   \item In Section~\ref{sec:Discrete} we consider the \textit{natural time discretization of the MLE}, by assuming that the first $N$ Fourier modes are measured at $M$ time grid points, uniformly spaced over the time interval $[0,T]$. We study the asymptotic properties of the proposed estimator when $N,T,M\to\infty$. In particular, we prove that the estimator is consistent if $N,M,T\to \infty$, or $N,M\to \infty$ while $T$ is fixed, and if
$T^2N^{\frac{4\beta}{d}-1}/M^2\rightarrow 0$, where $d$ is the space dimension, and $\beta$ is the power of the Laplacian. Moreover, if $4\beta<d$, then consistency holds true, when $N\to\infty$, when $M,T$ are fixed. This, in particular implies that to estimate efficiently the drift parameter it is enough to observe the Fourier modes at one instant of time - a result that agrees with recent discoveries in \cite{CialencoHuang2017,BibingerTrabs2017} where the solution is sampled in physical domain.  Under some additional technical assumptions on the growth rates of $N,M$ and $T$, we also prove that the proposed estimator is also asymptotically normal, with the same rate of convergence $\sqrt{T}N^{\frac{\beta}{d}+\frac{1}{2}}$ as the MLE from continuous time observation setup.
 \end{enumerate}
Some technical proofs, auxiliary results and relevant elements of Malliavin calculus are deferred to Appendix.

\smallskip\noindent
\textbf{Open problems and future work.}  A reasonable extension of the present work is to investigate similar estimators and problems given that the solution is observed discretely in the physical domain, in which case, one has to additionally approximate the Fourier modes by a sum. While analogous asymptotic properties are expected to hold true, rigourous proofs remain to be established. As already mentioned, most of the existing literature on parameter estimation for SPDEs is focused on sampling the Fourier modes in continuous time. In particular, the MLE approach was successfully applied to nonlinear equations \cite{IgorNathanAditiveNS2010,PasemannStannat2019}, and to equations driven by a fractional noise \cite{IgorSergeyJan2008}. Besides MLEs, in \cite{CialencoGongHuang2016} the authors propose an alternative class of estimators, called trajectory fitting estimators, and a Bayesian approach to estimating drift coefficients for a class of SPDEs driven by multiplicative noise is considered in \cite{ZCG2018}. It is imperative, from theoretical and practical point of view, to study the asymptotic properties of the discretized versions of the estimators proposed in the above mentioned works, especially by looking at various asymptotic regimes (large time-space sampling, small mesh size, etc).

\section{Setup of the problem and some auxiliary results}\label{sec:Setup}
Let $(\Omega,\sF,\set{\sF_t}_{t\geq 0}, \bP)$ be a stochastic basis with usual assumptions, and let $\set{w_j,\ j\geq 1}$ be a collection of independent standard Brownian motions on this basis. Assume that $G$ is a bounded and smooth domain in $\bR^d$, and let us denote by $\Laplace$ the Laplace operator on $G$ with zero boundary conditions. The corresponding scale of Sobolev spaces will be denoted by $H^s(G)$, or simply $H^s$, for $s\in\bR$.
It is well known (cf. \cite{Shubin}) that: a) the set $\set{h_k}_{k\in\bN}$ of eigenfunctions of $\Laplace$ forms a complete orthonormal system in $L^2(G)$; b) the corresponding eigenvalues $\nu_k,k\in\bN$, can be arranged such that $0<-\nu_1\leq - \nu_2\leq \ldots$, and there exists a positive constant $\varpi$ so that
\begin{equation*}
  \lim_{k\to\infty} |\nu_k|k^{-2/d} = \varpi.
\end{equation*}
In what follows, we will use the notation $\lambda_k:=\sqrt{-\nu_k}, \ k\in\bN$, and $\Lambda=\sqrt{-\Laplace}$. Also, for two sequences of numbers $\{a_n\}$ and $\{b_n\}$, we will write $a_n \sim b_n$, if there exists a nonzero and finite number $c$ such that $\lim_{n\to\infty}a_n/b_n=c$, and $a_n\simeq b_n$, if $\lim_{n\to\infty}a_n/b_n=1$.

We consider the following stochastic PDE
\begin{equation}\label{eq:mainSPDE}
\dif U(t,x) + \theta (-\Laplace)^\beta U(t,x)\dif t = \sigma \sum_{k\in\bN} \lambda_k^{-\gamma}h_k(x)\dif w_k(t), \quad t\in[0,T], \ U(0,x) = U_0, \ x\in G,
\end{equation}
where $\theta>0$, $\beta>0, \ \gamma \geq 0$, $\sigma>0$, and $U_0\in H^s(G)$ for some $s\in\bR$.

Using standard arguments (cf. \cite{ChowBook,LototskyRozovsky2017Book,RozovskyRozovsky2018Book}), it can be proved that if $2(\gamma-s)/d > 1$, then \eqref{eq:mainSPDE} has a unique  solution $U$, weak in PDE sense and strong in probability sense, such that
\begin{equation*}
U\in L^2(\Omega\times[0,T]; H^{s+\beta})\cap L^{2}(\Omega; C((0,T); H^s)).
\end{equation*}
In what follows, we will assume that $s\geq 0$, and $2\gamma > d$. We denote by $u_k,k\in\bN,$ the Fourier coefficient of the solution $U$ of \eqref{eq:mainSPDE} with respect to $h_k,k\in\bN$, i.e. $u_k(t) = (U(t),h_k)_0, k\in\bN$. Let $H^N$ be the finite dimensional subspace of $L^2(G)$ generated by $\set{h_k}_{k=1}^N$, and denote by $P_N$ the projection operator of $L^2(G)$ into $H^N$, and put $U^N = P_NU$, or equivalently $U^N:=(u_1,\ldots,u_N)$. Clearly, the Fourier mode $u_k,k\in\bN$, follows the dynamics of an Ornstein-Uhlenbeck process given by
\begin{equation*}
\dif u_k = -\theta \lambda_k^{2\beta} u_k\dif t + \sigma \lambda_k^{-\gamma} \dif w_k(t), \quad  u_k(0) = (U_0,h_k), \ t\geq 0.
\end{equation*}
We denote by $\bP^{N,T}_{\theta}$ the probability measure on $C([0,T]; H^N)\backsimeq C([0,T]; \bR^N)$ generated by the $U^N$.
The measures $\bP^{N,T}_{\theta}$ are equivalent for different values of the parameter $\theta$, and the Radon-Nikodym derivative, or likelihood ratio, has the form
\begin{align*}
\frac{\bP^{N,T}_{\theta}}{\bP^{N,T}_{\theta_{0}}} (U^N) =
\exp\left(-\frac{\left(\theta-\theta_0\right)}{\sigma^{2}} \sum_{k=1}^N\lambda_k^{2\beta+2\gamma}\int_0^T  u_k(t)\dif u_k(t)
- \frac{\left(\theta^2-\theta_0^2\right)}{2\sigma^{2}}\sum_{k=1}^{N}\lambda_k^{4\beta+2\gamma}\int_0^Tu_k^2(t)\dif t\right).
\end{align*}
By maximizing the log likelihood ratio with the respect  to the parameter of interest $\theta$,  we obtain the Maximum Likelihood Estimator (MLE) for $\theta$ given by
\begin{equation}\label{eq:MLE-UN}
\widehat{\theta}_{N,T} := -\frac{\sum_{k=1}^{N}\lambda_k^{2\beta+2\gamma}\int_0^T u_k(t)\dif u_k(t)}{\sum_{k=1}^{N}\lambda_k^{4\beta+2\gamma}\int_0^T u_k^2(t)\dif t},
\quad N\in\bN, \ T>0.
\end{equation}
Let us also compute the Fisher information related to $\bP^{N,T}_{\theta}/\bP^{N,T}_{\theta_{0}}$.
For simplicity, set $U_0=0$. Namely,
\begin{align*}
  \cI_{N,T} & := \int \left|  \frac{\partial}{\partial \theta} \log \frac{\dif \bP^{N,T}_{\theta}}{\dif \bP^{N,T}_{\theta_0}} \right|^2
  \left(\frac{\dif \bP^{N,T}_{\theta_0}}{\dif \bP^{N,T}_{\theta}}\right)^{-1} \dif \bP^{N,T}_{\theta_0} \nonumber\\
  & = - \int  \frac{\partial^2}{\partial \theta^2} \log \frac{\dif \bP^{N,T}_{\theta}}{\dif \bP^{N,T}_{\theta_0}}
  \left(\frac{\dif \bP^{N,T}_{\theta_0}}{\dif \bP^{N,T}_{\theta}}\right)^{-1} \dif \bP^{N,T}_{\theta_0}  \nonumber\\
  & = \frac{1}{\sigma^2} \sum_{k=1}^{N} \lambda_k^{4\beta +2\gamma} \bE\left[ \int_0^T u_k^2 \dif t \right].
\end{align*}
By direct evaluations, we have that
$$
\bE\left[ \int_0^T u_k^2 \dif t \right] = \frac{\sigma^2\lambda_k^{-2\gamma-2\beta}}{2\theta_0}\left(T-\frac{1-e^{-2\theta_0 \lambda_k^{2\beta}T}}{2\theta_0\lambda_k^{2\beta}}\right),
$$
which yields
\begin{align}
\cI_{N,T} & = \frac{1}{2\theta_0}\sum_{k=1}^{N}\lambda_k^{2\beta} \left(T-\frac{1-e^{-2\theta_0 \lambda_k^{2\beta}T}}{2\theta_0\lambda_k^{2\beta}}\right)
\simeq \frac{T}{2\theta_0}\sum_{k=1}^N \lambda_k^{2\beta},\quad \mbox{as}\ T\to \infty
\nonumber\\
& \simeq \frac{\varpi^{\beta}dTN^{\frac{2\beta}{d}+1}}{\left(4\beta+2d\right)\theta_0},\quad \mbox{as}\ N,T\to \infty .\label{eq:FisherInfo2}
\end{align}
In particular, note that $\cI_{N,T}\to\infty$, when $N,T\to\infty$.

\section{Asymptotics in large time and large number of Fourier modes}\label{sec:contTime}
It is known that the estimator $\wh \theta_{N,T}$ is unbiased, strongly consistent and asymptotically normal in two asymptotic regimes: $N\to\infty$ and $T$ fixed, and $T\to\infty$ and $N$ fixed; see for instance \cite{CXu2015,Cialenco2018} and references therein. In particular, for every fixed $T>0$,  $\lim_{N\to\infty} \theta_{N,T} = \theta_0$, with probability one, and
\begin{align}
\w\lim_{N\to\infty} N^{\beta/d+\frac{1}{2}} \left( \widehat{\theta}_{N,T} -\theta_0\right) &= \cN\left(0,\frac{(4\beta/d+2)\theta_0}{\varpi^{\beta}T}\right), \label{eq:AsymNormN}
\end{align}
where $\w\,\lim$ denotes the limit in distribution\footnote{Whenever convenient, we will also use the notation `$\overset{d}{\longrightarrow}$' to denote the convergence in distribution of random variables.}, and   $\cN(0,\bar{\sigma}^2)$ is a Gaussian random variable\footnote{Throughout the text we will use the notation $\cN(\mu_0,\sigma_0^2)$ to denote a Gaussian random variable with mean $\mu_0$ and variance $\sigma_0^2$.} with mean zero and variance $\bar\sigma^2$. Similarly, for every fixed $N\in\bN$, $\lim_{T\to\infty} \theta_{N,T} = \theta_0$, with probability one, and
\begin{equation}
 \w \lim_{T\to\infty} \sqrt{T} \left( \widehat{\theta}_{N,T} -\theta_0\right) = \cN(0,2\theta_0/J), \label{eq:AsymNormT}
\end{equation}
where $J=\sum_{k=1}^N\lambda_k^{2\beta}$.

To the best of our knowledge, the asymptotic properties of $\wh \theta_{N,T}$ when (both) $T,N\to\infty$ is not studied in the current literature. Besides this being an important question alone, the obtained results in this section will also serve as theoretical basis for investigating the statistical properties of the discretized version of the MLE studied later in this paper. In view of the above, naturally one should expect that the joint time-space consistency $\lim_{N,T\to\infty}\theta_{N,T} = \theta_0$ is satisfied. On the other hand, by \eqref{eq:AsymNormN}
\begin{align}
&\w\lim_{T\to\infty}\lim_{N\to\infty}\sqrt{T}N^{\beta/d+\frac{1}{2}}(\wh \theta_{N,T} - \theta_0)
=\cN\left(0,\frac{(4\beta/d+2)\theta_0}{\varpi^{\beta}}\right), \label{eq:Normality-TN}
\end{align}
and by \eqref{eq:AsymNormT} same result holds for the swaped limiting order $T\to\infty, N\to\infty$. While \eqref{eq:Normality-TN} does not have great statistical meaning, it leads to a reasonable ansatz that same identity \eqref{eq:Normality-TN} should be satisfied when both $N,T\to\infty$. Also note that, in view of  \eqref{eq:FisherInfo2}, this estimator is also optimal in the mean-square sense, having the rate of convergence dictated by the Fisher information. Next we give a rigourous proof of these results.

\begin{theorem}\label{th:NT1}
Assume that $\beta>1/2$ and $\gamma > d/2$. Then, $\wh \theta_{N,T}$ is strongly consistent, i.e.
\begin{equation}
\lim_{N,T \to\infty} \wh \theta_{N,T} = \theta_0,   \ \  \textrm{with probability one}, \label{eq:ConsNT}
\end{equation}
and asymptotically normal, i.e.
\begin{equation}
\w\lim_{N,T\rightarrow \infty} \sqrt{T}N^{\frac{\beta}{d}+\frac{1}{2}} \left(\widehat{\theta}_{N,T} -\theta_0\right)= \cN\left(0,\frac{(4\beta/d+2)\theta_0}{\varpi^{\beta}}\right). \label{eq:NTAsymNormal1}
\end{equation}
\end{theorem}

\begin{proof}
For simplicity, we set $U_0=0$, and hence $u_k(0)=0$ for all $k\geq 1$.  Since
\begin{equation}\label{eq:u_k}
u_k(t)=\sigma \lambda_k^{-\gamma} \int_0^t e^{-\theta_0 \lambda_k^{2\beta}(t-s)}\dif w_k(s),\ k\geq 1,
\end{equation}
it is straightforward to show that
\begin{align}\label{eq:2ndmomentuk}
\mathbb{E}u_k^{2}(t)= &
\sigma^2 \lambda_k^{-2\beta-2\gamma} \frac{\left(1-e^{-2\theta_0 \lambda_k^{2\beta}t}\right)}{2\theta_0}, \\
\label{eq:4thmomentuk}
\mathbb{E}u_k^{4}(t)= & 3 \sigma^4 \lambda_k^{-4\beta-4\gamma} \frac{\left(1-e^{-2\theta_0 \lambda_k^{2\beta}t}\right)^2}{(2\theta_0)^2}.
\end{align}
We note that
\begin{align}
\widehat{\theta}_{N,T}-\theta_0 & = -\frac{\sigma \sum_{k=1}^N \lambda_k^{2\beta+\gamma} \int_0^{T} u_k(t)\dif w_k(t)}
{\sum_{k=1}^N \lambda_k^{4\beta+2\gamma}\int_0^{T} u_k^2(t)\dif t} \nonumber\\
&= -\frac{\sigma \sum_{k=1}^N\xi_{k,T}}{\sum_{k=1}^N \Var\left(\xi_{k,T}\right)}
\cdot \frac{\sum_{k=1}^N \Var\left(\xi_{k,T}\right)} {\sum_{k=1}^N \lambda_k^{4\beta+2\gamma}\int_0^{T} u_k^2(t)\dif t}, \label{eq:thetahat-theta0}
\end{align}
where
$$
\xi_{k,T}:=\lambda_k^{2\beta+\gamma}\int_0^{T}u_k(t)\dif w_k(t).
$$
To show consistency \eqref{eq:ConsNT}, we will use  the strong law of large numbers \cite[Theorem IV.3.2]{ShiryaevBookProbability}.
From \eqref{eq:2ndmomentuk}, we have
$$
\Var(\xi_{k,T})=\lambda_k^{4\beta+2\gamma} \int_0^{T} \mathbb{E} u_k^2(t)\dif t
=\sigma^2 \lambda_k^{2\beta} \int_0^{T} \frac{1-e^{-2\theta_0 \lambda_k^{2\beta}t}}{2\theta_0} \dif t \simeq \frac{\sigma^2\lambda_k^{2\beta}T}{2\theta_0},\quad \mbox{as}\ T\to \infty,
$$
and thus,
\begin{align}
\sum_{k=1}^N \Var(\xi_{k,T}) &\simeq \frac{\sigma^2T}{2\theta_0}\sum_{k=1}^N \lambda_k^{2\beta},\quad \mbox{as}\ T\to \infty \nonumber \\
&\simeq \frac{\sigma^2 \varpi^{\beta}dTN^{\frac{2\beta}{d}+1}}{(4\beta+2d)\theta_0},
\quad \mbox{as}\ N,T\to \infty.\label{asympSumVarxi}
\end{align}
Moreover, using \eqref{eq:4thmomentuk}, we get that
\begin{equation}
\label{asympVar}
\begin{split}
\Var\left(\lambda_k^{4\beta+2\gamma}\int_0^T u_k^2(t)\dif t\right)
&\leq \mathbb{E}\left(\lambda_k^{4\beta+2\gamma}\int_0^T u_k^2(t)\dif t\right)^2 \leq \lambda_k^{8\beta+4\gamma}T\int_0^T\mathbb{E}u_k^4(t)\dif t\\
&\sim \lambda_k^{4\beta}T^2,\quad \mbox{as}\ T\to \infty.
\end{split}
\end{equation}
Hence, there exists $T_0>0$ such that for all $T\geq T_0$,
\begin{align*}
&\sum_{N=1}^{\infty}\frac{ \Var\left(\xi_{N,T}\right)}{\left(\sum_{k=1}^N \Var\left(\xi_{k,T}\right)\right)^2}\leq
\frac{C_1}{T}\sum_{N=1}^{\infty} \frac{\lambda_N^{2\beta}}{\left(
\sum_{k=1}^N \lambda_k^{2\beta}\right)^2}\leq \frac{C_2}{T}
\sum_{N=1}^{\infty}\frac{1}{N^{2+\frac{2\beta}{d}}}\leq C_3<\infty, \\
&\sum_{N=1}^{\infty}
\frac{\Var\left(\lambda_N^{4\beta+2\gamma}\int_0^T u_N^2(t)\dif t\right)}
{\left(\sum_{k=1}^N \Var\left(\xi_{k,T}\right)\right)^2}
\leq C_4 \sum_{N=1}^{\infty}\frac{\lambda_N^{4\beta}}{\left(
\sum_{k=1}^N \lambda_k^{2\beta}\right)^2}
\leq C_5\sum_{N=1}^{\infty}\frac{1}{N^2}<\infty,
\end{align*}
where $C_1,C_2,C_3,C_4,C_5>0$ are some constants\footnote{Notoriously, we will denote by $C$ with subindexes generic constants that may change from line to line. } independent of $T$. Using the uniform boundedness of the above series, and employing the strong law of large numbers, we deduce that for every $\varepsilon>0$ and $T\geq T_0$, there exists $N_0>0$ independent of $T$ such that for $N\geq N_0$,
$$
\left|\frac{\sigma \sum_{k=1}^N\xi_{k,T}}{\sum_{k=1}^N \Var\left(\xi_{k,T}\right)}  \right|< \varepsilon,  \quad 
\mbox{and} \quad
\left|\frac{\sum_{k=1}^N \Var\left(\xi_{k,T}\right)}
{\sum_{k=1}^N \lambda_k^{4\beta+2\gamma}\int_0^{T} u_k^2(t)\dif t}-1\right|
< \varepsilon
$$
with probability one. Therefore,
\begin{equation}\label{strongLLN}
\lim_{N,T \to\infty} \frac{\sigma \sum_{k=1}^N\xi_{k,T}}{\sum_{k=1}^N \Var\left(\xi_{k,T}\right)} =0 \quad \mbox{and} \quad
\lim_{N,T \to\infty}\frac{\sum_{k=1}^N \Var\left(\xi_{k,T}\right)} {\sum_{k=1}^N \lambda_k^{4\beta+2\gamma}\int_0^{T} u_k^2(t)\dif t} = 1
\end{equation}
with probability one. From here, and using \eqref{eq:thetahat-theta0}, the proof of \eqref{eq:ConsNT} is complete.

Next, we will prove asymptotic normality property \eqref{eq:NTAsymNormal1}, starting with representation
\begin{equation}
\label{thetaNT-theta}
\widehat{\theta}_{N,T}-\theta_0 = -\frac{\sigma \sum_{k=1}^N \xi_{k,T}}
{\left(\sum_{k=1}^N \Var(\xi_{k,T})\right)^{1/2}} \cdot
\frac{1}{\left(\sum_{k=1}^N \Var(\xi_{k,T})\right)^{1/2}}\cdot
\frac{\sum_{k=1}^N \Var(\xi_{k,T})}{\sum_{k=1}^N \lambda_k^{4\beta+2\gamma}\int_0^{T} u_k^2(t)\dif t}.
\end{equation}
Let us consider the first term in \eqref{thetaNT-theta}. We will show that
\begin{equation*}
\w\lim_{N,T\to\infty}\frac{\sigma \sum_{k=1}^N \xi_{k,T}}
{\left(\sum_{k=1}^N \Var(\xi_{k,T})\right)^{1/2}} = \mathcal{N}(0,\sigma^2).
\end{equation*}
By Burkholder--Davis--Gundy inequality and Cauchy--Schwartz inequality, we have
\begin{align*}
\mathbb{E}\xi_{k,T}^4&=\mathbb{E}\left(\lambda_k^{2\beta+\gamma}\int_0^{T}u_k(t)\dif w_k(t)
\right)^4\leq C_1\lambda_k^{8\beta+4\gamma}\mathbb{E}\left(\int_0^{T}u_k^2(t)\dif t\right)^2 \nonumber\\
&\leq C_1\lambda_k^{8\beta+4\gamma} T \int_0^{T}\mathbb{E}u_k^4(t)\dif t, 
\end{align*}
for some $C_1>0$. By \eqref{eq:4thmomentuk} and \eqref{asympVar},
there exists $T_1>0$ such that for all $T\geq T_1$,
$\mathbb{E}\xi_{k,T}^4\leq C_2\lambda_k^{4\beta}T^2$, for some $C_2>0$, and hence
there exists $N_1>0$ independent of $T$ and $T_1$ such that for all $N\geq N_1$ and
for all $T\geq T_1$,
\begin{equation*}
\sum_{k=1}^N \mathbb{E}\xi_{k,T}^4\leq C_3N^{\frac{4\beta}{d}+1}T^2,
\end{equation*}
for some  $C_3>0$, independent of $N$, and $T$.
We will verify the classical Lindeberg condition \cite[Theorem III.5.1]{ShiryaevBookProbability}, namely that for every $\varepsilon>0$,
\begin{equation*}
\lim_{N,T\rightarrow \infty} \frac{\sum_{k=1}^N \mathbb{E}\left(
\xi_{k,T}^21_{\left\{|\xi_{k,T}|>\varepsilon \sqrt{\sum_{k=1}^N \Var(\xi_{k,T})}\right\}}
\right)}{\sum_{k=1}^N \Var(\xi_{k,T})}=0.
\end{equation*}
By Cauchy-Schwartz inequality and Chebyshev inequality,
\begin{align*}
\sum_{k=1}^N \mathbb{E}\left(
\xi_{k,T}^21_{\left\{|\xi_{k,T}|>\varepsilon \sqrt{\sum_{k=1}^N \Var(\xi_{k,T})}\right\}}\right)
&\leq \sum_{k=1}^N
\left(\mathbb{E}\xi_{k,T}^4\right)^{1/2}\left(\mathbb{P}\left(
|\xi_{k,T}|>\varepsilon \sqrt{\sum_{k=1}^N \Var(\xi_{k,T})}
\right)\right)^{1/2}\\
&\leq \frac{\sum_{k=1}^N
\mathbb{E}\xi_{k,T}^4}{\varepsilon^2 \sum_{k=1}^N \Var(\xi_{k,T})}.
\end{align*}
Consequently,
$$
\frac{\sum_{k=1}^N \mathbb{E}\left(
\xi_{k,T}^21_{\left\{|\xi_{k,T}|>\varepsilon \sqrt{\sum_{k=1}^N \Var(\xi_{k,T})}\right\}}
\right)}{\sum_{k=1}^N \Var(\xi_{k,T})}\leq
\frac{\sum_{k=1}^N
\mathbb{E}\xi_{k,T}^4}{\varepsilon^2 \left(\sum_{k=1}^N \Var(\xi_{k,T})\right)^2}
\sim \frac{1}{\varepsilon^2 N},\quad \mbox{as}\ N,T \to \infty.
$$
Thus,
\begin{equation}\label{CLT-xikT}
\frac{\sum_{k=1}^N \xi_{k,T}}{\left(\sum_{k=1}^N \Var(\xi_{k,T})\right)^{1/2}}\xrightarrow[N,T\to\infty]{d}  \mathcal{N}(0,1).
\end{equation}
We also note that
$$
\left(\sum_{k=1}^N \lambda_k^{4\beta+2\gamma}\int_0^{T} \mathbb{E} u_k^2(t)\dif t\right)^{1/2}
\simeq \frac{\sigma \sqrt{\varpi^{\beta}d} \sqrt{T}N^{\frac{\beta}{d}+\frac{1}{2}}}{\sqrt{(4\beta+2d)\theta_0}}\quad
\mbox{as}\ N,T\to \infty.
$$
In view of the strong law of large numbers,
\begin{equation*}
\lim_{N,T\to\infty}\frac{\sum_{k=1}^N \lambda_k^{4\beta+2\gamma}\int_0^{T} u_k^2(t)\dif t}
{\sum_{k=1}^N \lambda_k^{4\beta+2\gamma}\int_0^{T} \mathbb{E} u_k^2(t)\dif t}=1
\end{equation*}
with probability one. Finally, combining all the above and using Slutzky's theorem, \eqref{eq:NTAsymNormal1} follows at once.
This completes the proof.
\end{proof}

\subsection{Asymptotic normality of the MLE by Malliavin-Stein's approach}\label{sec:MalCal}
In this section, we give an alternative proof of \eqref{eq:NTAsymNormal1} using tools and results from Malliavin calculus. While this method of proof has its stand along value, we also believe that it will serve as a theoretical base for future studies related to other discretized estimators and different asymptotic regimes.

As before, let $U_0=0$, and for convenience, in this section we will use the following notations:
\begin{align*}
F_{N,T}&:= \wh \theta_{N,T}-\theta_0 = -\frac{\sigma \sum_{k=1}^N \lambda_k^{2\beta+\gamma} \int_0^{T} u_k(t)\dif w_k(t)}
{\sum_{k=1}^N \lambda_k^{4\beta+2\gamma}\int_0^{T} u_k^2(t)\dif t}  =:
-\frac{\boldsymbol{F}_1(N,T)}{\boldsymbol{F}_2(N,T)},\\
\wh{F}_{N,T} & := -\frac{\sigma \sum_{k=1}^N \lambda_k^{2\beta+\gamma} \int_0^{T} u_k(t)\dif w_k(t)}
{\E\left[\sum_{k=1}^N \lambda_k^{4\beta+2\gamma}\int_0^{T} u_k^2(t)\dif t\right]} = -\frac{\boldsymbol{F}_1(N,T)}{C_{N,T}^2},
\end{align*}
where $C_{N,T}^2:=\E\left(\boldsymbol{F}_2(N,T)\right)$.
We note that $\wh F_N$ can be written as a double stochastic integral and in view of  \cite[Theorem 2.7.7]{NourdinPeccati2012}, $\wh F_N$ belongs to the second-order chaos; see also Appendix~\ref{apendix:MalCal}. Next we present a key technical result.

\begin{lemma}\label{conv-variance}
Let $\cH$ be the space endowed with the inner product defined in \eqref{innerprod-1}.
Let $D$ be the Malliavin derivative defined in \eqref{def-MD-tj}.
Then, we have
\begin{align*}
\sqrt{\Var\left(\frac{1}{2}\| C_{N,T} D\wh{F}_{N,T}\|_{\mathcal{H}}^2\right)}\longrightarrow 0, \mbox{ as } N,T \rightarrow \infty.
\end{align*}
\end{lemma}
\noindent The proof of Lemma~\ref{conv-variance} is deferred to Appendix~\ref{Appendix:proofs}.

To prove asymptotic normality of $\widehat{\theta}_{N,T}$, we will show that $C_{N,T}F_{N,T}\overset{d}{\longrightarrow} \cN(0,\sigma^2)$, as $N,T\to \infty$. The variance $\sigma^2$ comes from the fact $\E\left(C^2_{N,T}\wh{F}_{N,T}^2\right)=\si^2.$
We split $C_{N,T}F_{N,T}$ into
\begin{equation}\label{eq:M1}
C_{N,T}F_{N,T}=C_{N,T}\left(F_{N,T}-\widehat{F}_{N,T}\right)+C_{N,T}\widehat{F}_{N,T}.
\end{equation}
We note that
$$
 C_{N,T}(F_{N,T}-\hat F_{N,T}) = \frac{C_{N,T}^2}{\boldsymbol{F}_2(N,T)} \frac{\boldsymbol{F}_1(N,T)}{C_{N,T}}
\left(1-\frac{\boldsymbol{F}_2(N,T)}{C_{N,T}^2} \right).
$$
From \eqref{strongLLN},
$$
\frac{C_{N,T}^2}{\boldsymbol{F}_2(N,T)} \longrightarrow 1, \quad
1-\frac{\boldsymbol{F}_2(N,T)}{C_{N,T}^2} \longrightarrow 0, \quad
\mbox{as}\ N,T \rightarrow \infty
$$
with probability 1. On the other hand, by Lemma~\ref{conv-variance}, we have
$$
\frac{\boldsymbol{F}_1(N,T)}{C_{N,T}}=C_{N,T}\widehat F_{N,T}\overset{d}{\longrightarrow} \cN(0,\sigma^2),
\quad \mbox{as}\ N,T\to \infty.
$$
Hence, by Slutzky's theorem, we deduce that $C_{N,T}(F_{N,T}-\widehat F_{N,T}) \overset{d}{\longrightarrow} 0$ , as $N,T\to \infty$, which consequently implies that
\begin{equation}\label{eq:M2}
C_{N,T}(F_{N,T}-\widehat F_{N,T}) \longrightarrow 0, \quad \mbox{ as } N,T
\to \infty,\quad\mbox{ in probability}.
\end{equation}

To deal with the second term in \eqref{eq:M1}, we note that by Lemma~\ref{conv-variance} and Proposition~\ref{Th2.2}, we get that
$$
\lim_{N,T\to\infty}d_{TV}\left(C_{N,T}\wh{F}_{N,T}, \cN(0,\sigma^2)\right)=0.
$$
Consequently, Theorem~\ref{Equiv} implies that
$$
\w\lim_{N,T\to\infty}C_{N,T}\widehat{F}_{N,T}= \cN\left(0,\sigma^2\right).
$$
This, combined with \eqref{eq:M1} and \eqref{eq:M2}, implies that
$$
\w\lim_{N,T\to\infty}C_{N,T} F_{N,T}= \cn(0,\sigma^2).
$$
Finally, note that, by \eqref{asympSumVarxi}
$$
\frac{C_{N,T}}{\sigma} \simeq \frac{\sqrt{\varpi^\beta d}  \sqrt{T} N^{\frac{\beta}{d}+\frac{1}{2}}}{(4\beta+2d)\theta_0}, \quad \mbox{as}\ N,T\to \infty,
$$
which implies \eqref{eq:NTAsymNormal1}, and the proof is complete.

%%%%%%%%%%%%%%%%%%%%%%%%%%%%%%%%%%%%%%%%%%%%%%%%%%%%%%%%%%%

\section{Asymptotic properties of the discretized MLE}\label{sec:Discrete}

In this section, we investigate statistical properties of the discretized version of MLE \eqref{eq:MLE-UN}.
Towards this end, we assume that the Fourier modes $u_k(t)$, $k\geq 1$, are observed at a uniform time grid
$$
0=t_0<t_1<\cdots < t_M =T,\quad \textrm{with}\ \Delta t:=t_i-t_{i-1}=\displaystyle\frac{T}{M},\
i=1,\ldots,M.
$$
We consider the discretized MLE $\widetilde{\theta}_{N,M,T}$ defined by
\begin{equation*}
\widetilde{\theta}_{N,M,T}:=-\frac{\sum_{k=1}^N \lambda_k^{2\beta+2\gamma}\sum_{i=1}^M
u_k(t_{i-1})\left[u_k(t_i)-u_k(t_{i-1})\right]}
{\sum_{k=1}^N \lambda_k^{4\beta+2\gamma}\sum_{i=1}^M u_k^2(t_{i-1})\Delta t}.
\end{equation*}
We are interested in studying the asymptotic properties of $\widetilde{\theta}_{N,M,T}$, as $N,M,T\to\infty$.

For simplicity of writing, we also introduce the following notations:
\begin{align*}
Y_{N,M,T}&:=\sum_{k=1}^N \lambda_k^{2\beta+\gamma}\sum_{i=1}^M u_k(t_{i-1})\left(w_k(t_i)-w_k(t_{i-1})\right), &
Y_{N,T}&:=\sum_{k=1}^N \lambda_k^{2\beta+\gamma}\int_0^T u_k(t)\dif w_k(t),\\
I_{N,M,T}&:=\sum_{k=1}^N \lambda_k^{4\beta+2\gamma}\sum_{i=1}^M u_k^2(t_{i-1})\Delta t,
& I_{N,T} & :=\sum_{k=1}^N \lambda_k^{4\beta+2\gamma}\int_0^T u_k^2(t)\dif t,\\
V_{N,M,T}&:=\sum_{k=1}^N \lambda_k^{4\beta+2\gamma} \sum_{i=1}^M u_k(t_{i-1})
\int_{t_{i-1}}^{t_i} \left(u_k(t)-u_k(t_{i-1})\right)dt, &  \Upsilon&:=  \left(\frac{\varpi^\beta}{(4\beta/d+2)\theta_0}\right)^{1/2}.
\end{align*}
A key step in the proofs of the main results is to write $\widetilde{\theta}_{N,M,T}$ as
\begin{equation}\label{Dis-MLE-theta}
\widetilde{\theta}_{N,M,T}-\theta_0 = \frac{\theta_0 V_{N,M,T}}{I_{N,M,T}} -\frac{\sigma Y_{N,M,T}}{I_{N,M,T}}.
\end{equation}

Below we present several important technical results, and to streamline the presentation their  proofs are deferred to Appendix~\ref{Appendix:proofs}.

\begin{lemma}\label{lemma:Discr1}
For $0<t<s\leq T$ and $k, l\in \mathbb{N}$, we have that
\begin{align}
\mathbb{E}(u_k(t)u_k(s))& =\frac{\sigma^2\lambda_k^{-2\gamma-2\beta}}{2\theta_0}\left[e^{-\theta_0 \lambda_k^{2\beta}(s-t)}-e^{-\theta_0 \lambda_k^{2\beta} (s+t)}\right], \label{eq:productuk} \\
\mathbb{E}|u_k(t)-u_k(s)|^{2l}&\leq C(l)\left(\sigma^2\lambda_k^{-2\gamma}\right)^l |t-s|^l, \label{eq:continuityuk} \\
\mathbb{E}|u_k(t)+u_k(s)|^{2l}&\leq \bar C(l)\left(\sigma^2\lambda_k^{-2\gamma-2\beta}T\right)^l, \label{eq:bddnessuk}
\end{align}
for some $C(l), \bar C(l)>0$.
\end{lemma}

\begin{lemma}\label{lemma:Y-Difference}
For each $T>0,\ N,M\in \mathbb{N}$,
there exist $C>0$ independent of $N,M,T$ such that
\begin{align}
\mathbb{E}|Y_{N,M,T}-Y_{N,T}|^2 &\leq C \frac{T^2 N^{\frac{4\beta}{d}+1}}{M}, \label{eq:Y-NMT}\\
\mathbb{E}\left|I_{N,M,T}-I_{N,T}\right|^2 &\leq C \frac{T^4N^{\frac{8\beta}{d}+1}}{M^2}, \label{eq:I-NMT}\\
\mathbb{E}|V_{N,M,T}|^2&\leq C\frac{T^4N^{\frac{8\beta}{d}+1}}{M^2}. \label{eq:V-NMT}
\end{align}
\end{lemma}

\begin{remark}
As a direct consequence of \eqref{asympSumVarxi} and \eqref{strongLLN}, we have that
\begin{equation}\label{eq:strongLLN-1}
\lim_{T,N \to\infty} \frac{Y_{N,T}}{TN^{\frac{2\beta}{d}+1}} =0, \quad \mbox{and} \quad
\lim_{T,N \to\infty}\frac{I_{N,T}} {\sigma^2 \Upsilon ^2 TN^{\frac{2\beta}{d}+1}}= 1,
\end{equation}
with probability one. Moreover, following the lines of the proof of \eqref{strongLLN}, one can show that
for every fixed $T>0$,
\begin{equation}\label{eq:strongLLN-2}
\lim_{N \to\infty} \frac{Y_{N,T}}{TN^{\frac{2\beta}{d}+1}} =0, \quad \mbox{and} \quad
\lim_{N \to\infty}\frac{I_{N,T}} {\sigma^2 \Upsilon ^2 TN^{\frac{2\beta}{d}+1}}= 1,
\end{equation}
with probability one.
\end{remark}

\smallskip\noindent
With these at hand, we are ready to show that $\widetilde{\theta}_{N,M,T}$ is a weakly consistent estimator of $\theta_0$.
\begin{theorem}\label{th:DiscrtConsistency}
Assume that $\beta>1/2$ and $\gamma > d/2$.
Then,
\begin{equation}\label{eq:thetaNTM-Consist}
\widetilde{\theta}_{N,M,T} \rightarrow \theta_0, \quad \mbox{in probability},
\end{equation}
as $N,M,T\to \infty$, or as $N,M\to \infty$ while $T$ is fixed, and assuming that (in both cases)
\begin{equation}\label{eq:MNT-Consist}
\frac{T^2N^{\frac{4\beta}{d}-1}}{M^2}\rightarrow 0.
\end{equation}
Moreover, if $4\beta<d$, then \eqref{eq:thetaNTM-Consist} holds true with both $T,M$ being fixed.
\end{theorem}

\begin{proof}
Let $\bar L := \mathbb{P}\left(\left|\widetilde{\theta}_{N,M,T}-\theta_0\right|>\varepsilon\right)$.  In view of \eqref{Dis-MLE-theta}, we note that
\begin{align*}
\bar L
&\leq \mathbb{P}\left(\left|\frac{\sigma Y_{N,M,T}}{I_{N,M,T}}\right|>\varepsilon/2\right)
+\mathbb{P}\left(\left|\frac{\theta_0V_{N,M,T}}{I_{N,M,T}}\right|>\varepsilon/2\right).
\end{align*}
Consequently, for an arbitrary fixed $\delta\in(0,\varepsilon/2)$, we have, using \eqref{eq:simpleOne},
\begin{align*}
\bar L
&\leq \mathbb{P}\left(\frac{\theta_0|V_{N,M,T}|}{\sigma^2\Upsilon^2 TN^{\frac{2\beta}{d}+1}}>\delta\right)
+\mathbb{P}\left(\frac{|Y_{N,M,T}|}{\sigma \Upsilon^2 TN^{\frac{2\beta}{d}+1}}>\delta\right) \\
&\qquad+2\mathbb{P}\left(\left|\frac{I_{N,M,T}}{\sigma^2\Upsilon^2TN^{\frac{2\beta}{d}+1}}-1\right|>\frac{\varepsilon-2\delta}{\varepsilon}\right)\\
& =: L_1 +L_2 +2L_3.
\end{align*}
By Chebyshev inequality\footnote{Throughout we will use the following version of the Chebyshev inequality: $\bP(|X|>a)\leq \bE(X^2)/a^2$, for $a>0$. }, and \eqref{eq:V-NMT}, we have that, for some constant (that may depend on $\delta$) $C_1(\delta)>0$,
$$
L_1 \leq C_1(\delta)\frac{T^2N^{\frac{4\beta}{d}-1}}{M^2}.
$$
As far as $L_2$, we write
\begin{align*}
L_2 \leq \mathbb{P}\left(
\frac{|Y_{N,M,T}-Y_{N,T}|}{\sigma \Upsilon^2 TN^{\frac{2\beta}{d}+1}}>\delta/2\right)+
\mathbb{P}\left(\frac{|Y_{N,T}|}{\sigma \Upsilon^2TN^{\frac{2\beta}{d}+1}}>\delta/2\right) =: L_{21}+L_{22}.
\end{align*}
Again by Chebyshev inequality, and using \eqref{eq:Y-NMT}, we get $L_{21} \leq C_2(\delta)/(NM)$, for some $C_2(\delta)>0$. On the other hand, by \eqref{eq:strongLLN-1}, $L_{22}\to 0$, as $N,T\to \infty$. Moreover, by  \eqref{eq:strongLLN-2} $L_{22}\to 0$, as $N\to \infty$ with $T$ being fixed.
We treat $L_3$ similarly:
\begin{align*}
L_3\leq \mathbb{P}\left(\frac{|I_{N,M,T}-I_{N,T}|}{\sigma^2 \Upsilon^2TN^{\frac{2\beta}{d}+1}}
>\frac{\varepsilon-2\delta}{2\varepsilon}\right)
+\mathbb{P}\left(\left|\frac{I_{N,T}}{\sigma^2\Upsilon^2 TN^{\frac{2\beta}{d}+1}}-1\right|>\frac{\varepsilon-2\delta}{2\varepsilon}\right) =: L_{31}+L_{32}.
\end{align*}
In view of  \eqref{eq:I-NMT}, and Chebyshev inequality, we have the bound
\begin{equation*}
L_{31}\leq C_3(\varepsilon) \frac{T^2N^{\frac{4\beta}{d}-1}}{M^2}.
\end{equation*}
By \eqref{eq:strongLLN-1}, and respectively \eqref{eq:strongLLN-2}, we get that $L_{32}\to0$, as $N,T\to \infty$, and, respectively, as $N\to \infty$ while $T$ fixed.
Hence, combining all the above bounds, we conclude that
\begin{equation*}
\bar L \leq C(\varepsilon)\left(\frac{T^2N^{\frac{4\beta}{d}-1}}{M^2} +
\frac{1}{NM}\right).
\end{equation*}
Clearly, $\bar L\to 0$ for every $\varepsilon>0$, as $N,T\rightarrow \infty$, or when $T$ fixed and as $N\to \infty$, given that \eqref{eq:MNT-Consist} is satisfied.
This concludes the proof.
\end{proof}

Next we prove an asymptotic  normality result for discretized MLE $\widetilde{\theta}_{N,M,T}$.  As one may expect, the rate of convergence of $\widetilde{\theta}_{N,M,T}$ agrees with those from  continuous time setup, and thus asymptotically is optimal in the mean-square sense.  As usual, we will denote by $\Phi$ the cumulative probability function of a standard Gaussian random variable.

\begin{theorem}\label{th:DiscrAsymNorm}
Assume that $\beta>1/2$ and $\gamma > d/2$. Then,
\begin{equation}\label{eq:DiscrtAsymNorm1}
\sup_{x\in \mathbb{R}}\left|
\mathbb{P}\left(\Upsilon \sqrt{T}N^{\frac{\beta}{d}+\frac{1}{2}}\left(\theta_0-\widetilde{\theta}_{N,M,T}\right)\leq x\right)-\Phi(x)
\right|\rightarrow 0,
\end{equation}
as $N,M,T\to \infty$,  or as $N,M\to \infty$ while $T$ is fixed, and such that (in both cases)
\begin{equation}\label{eq:MNTL}
\frac{T^3N^{6\beta/d}}{M^{2}}\rightarrow 0, \qquad \textrm{or} \qquad  \frac{TN^{2\beta/d}}{M}\rightarrow 0.
\end{equation}
\end{theorem}

\begin{proof}
We denote the left hand side of \eqref{eq:DiscrtAsymNorm1} by $\bar K$, and using \eqref{Dis-MLE-theta}, we write it as
\begin{equation*} 
\bar K = \sup_{x\in \mathbb{R}}\left|
\mathbb{P}\left(\sigma \Upsilon \sqrt{T}N^{\frac{\beta}{d}+\frac{1}{2}}\frac{\sigma Y_{N,M,T} - \theta_0 V_{N,M,T}}{I_{N,M,T}}
 \leq x\right)-\Phi(x)\right|.
\end{equation*}
Using \eqref{eq:simple3}, we continue
\begin{align*}
\bar K & \leq
\sup_{x\in \mathbb{R}}\left|
\mathbb{P}\left(\frac{\sigma Y_{N,M,T} - \theta_0 V_{N,M,T}}{\sigma \Upsilon \sqrt{T}N^{\frac{\beta}{d}+\frac{1}{2}}}
 \leq x\right)-\Phi(x)\right| +
\mathbb{P}\left(\left|\frac{I_{N,M,T}}{\sigma^2 \Upsilon^2 TN^{\frac{2\beta}{d}+1}}-1\right|>\varepsilon \right)+\varepsilon  \\
& =: K_1 + K_2 +\varepsilon.
\end{align*}
Consequently, by applying \eqref{eq:simple2} , we obtain
\begin{align*}
  K_1 & \leq \sup_{x\in \mathbb{R}}\left| \mathbb{P}\left(\frac{Y_{N,T}}{\sigma \Upsilon \sqrt{T}N^{\frac{\beta}{d}+\frac{1}{2}}}  \leq x\right)-\Phi(x) \right|
+\mathbb{P}\left(\frac{|Y_{N,M,T}-Y_{N,T}|}{\sigma \Upsilon \sqrt{T}N^{\frac{\beta}{d}+\frac{1}{2}}}
 >\varepsilon \right) \\
 & \qquad +  \mathbb{P}\left(\frac{\theta_0|V_{N,M,T}|}{\sigma \Upsilon \sqrt{T}N^{\frac{\beta}{d}+\frac{1}{2}}}>\varepsilon \right) + 2\varepsilon
 =: K_{1,1} + K_{1,2} + K_{1,3} + 2\varepsilon.
\end{align*}
Note that  \eqref{CLT-xikT} implies that
\begin{equation*}
\w\lim_{N,T\to \infty} \frac{Y_{N,T}}
{\sigma\Upsilon\sqrt{T}N^{\frac{\beta}{d}+\frac{1}{2}}} =\mathcal{N}(0,1).
\end{equation*}
Moreover, using \eqref{eq:AsymNormN}, one can easily derive that  the above limit holds true for every fixed $T>0$ and $N\to\infty$.
Thus, $K_{1,1}\to 0$, as $N,T\rightarrow \infty$ (or as $N\to \infty$, when $T$ is fixed).
By Chebyshev inequality and by \eqref{eq:Y-NMT} and  \eqref{eq:V-NMT}, we deduce
\begin{equation*}
K_{1,2}  \leq C_1(\varepsilon) \frac{TN^{\frac{2\beta}{d}}}{M}, \qquad K_{1,3} \leq C_2(\varepsilon) \frac{T^{3} N^{\frac{6\beta}{d}}}{M^2},
\end{equation*}
for some $C_1(\varepsilon), C_2(\varepsilon)>0$.

Similarly,
\begin{align*}
  K_2 & \leq \mathbb{P}\left(\left|\frac{I_{N,T}}{\sigma^2\Upsilon^2TN^{\frac{2\beta}{d}+1}} - 1\right|>\frac{\varepsilon}{2} \right)
+\mathbb{P}\left(\frac{|I_{N,M,T}-I_{N,T}|}{\sigma^2\Upsilon^2TN^{\frac{2\beta}{d}+1}}
>\frac{\varepsilon}{2} \right)=: K_{2,1}+K_{2,2}.
\end{align*}
By \eqref{eq:strongLLN-1} or \eqref{eq:strongLLN-2}, $K_{2,1}\to0$, as $N,T\to\infty$,  or as $N\to\infty$ and $T$ fixed. On the other hand, by \eqref{eq:I-NMT}, $$
K_{2,1} \leq C_3(\varepsilon) \frac{T^2 N^{\frac{4\beta}{d}-1}}{M^2}.
$$
Combining all the above, we conclude
$$
\bar K \leq C_4(\varepsilon)\left(\frac{TN^{\frac{2\beta}{d}}}{M} + \frac{T^{3} N^{\frac{6\beta}{d}}}{M^2} + \frac{T^2 N^{\frac{4\beta}{d}-1}}{M^2}\right) \leq C_5(\varepsilon) \left(\frac{TN^{\frac{2\beta}{d}}}{M} + \frac{T^{3} N^{\frac{6\beta}{d}}}{M^2}\right) +3\varepsilon.
$$
Since $\varepsilon>0$ was chosen arbitrarily, and since $C_5(\varepsilon)$ is independent of $N,M,T$, and given that \eqref{eq:MNTL} is satisfied, we conclude that $\bar K\to 0$, as $N,M,T\to \infty$, or $N,M\to\infty$ and $T$ being fixed. The proof is complete.
\end{proof}

\appendix
\section{Proofs of technical lemmas}\label{Appendix:proofs}

\begin{proof}[Proof of Lemma~\ref{conv-variance}]
We start by computing the Malliavin derivative of  $\wh{F}_{N,T}$. If $r\le t$ and for $1\le k\le N$, then
\begin{align*}
D_{r,k}u_k(t) &= \sigma \lambda_k^{-\gamma} D_{r,k}  \int_0^t e^{-\theta_0 \lambda_k^{2\beta}(t-s)}\dif w_k(s)\\
&= \sigma \lambda_k^{-\gamma}   e^{-\theta_0 \lambda_k^{2\beta}(t-r)}.
\end{align*}
Moreover, one has that $D_{r,j}u_k(t)=0$ if $j\ne k$ or $r>t$.
Therefore, for $r\le T$ and $1\le j\le N$, we have by \eqref{chain-rule-formula},
\begin{align}
D_{r,j}\wh{F}_{N,T}&= -\frac{\si}{C_{N,T}^2}  \lambda_j^{2\beta+\gamma} u_j(r)- \frac{\si}{C_{N,T}^2}  \sum_{k=1}^N \lambda_k^{2\beta+\gamma} \int_r^{T} D_{r,j}u_k(t)\dif w_k(t) \nonumber\\
&= -\frac{\si}{C_{N,T}^2}  \lambda_j^{2\beta+\gamma} u_j(r)-\frac{\si}{C_{N,T}^2}   \lambda_j^{2\beta+\gamma} \int_r^{T} D_{r,j}u_j(t)\dif w_j(t) \nonumber\\
&= -\frac{\si}{C_{N,T}^2}  \lambda_j^{2\beta+\gamma} u_j(r)- \frac{\si^2}{C_{N,T}^2}   \lambda_j^{2\beta} \int_r^{T}  e^{-\theta_0 \lambda_j^{2\beta}(t-r)} \dif w_j(t). \label{est1}
\end{align}
We continue by setting
$$
A:=\left\| C_{N,T} D\wh{F}_{N,T}\right\|_{\mathcal{H}}^2 =  C_{N,T}^2\| D\wh{F}_{N,T}\|_{\mathcal{H}}^2,
$$
and in view of \eqref{est1}, we obtain
\begin{align*}
A&=C_{N,T}^2\int_0^T \sum_{k=1}^N \left[\frac{\si}{C_{N,T}^2}  \lambda_k^{2\beta+\gamma} u_k(r)+
\frac{\si^2}{C_{N,T}^2} \lambda_k^{2\beta} \int_r^{T}   e^{-\theta_0 \lambda_k^{2\beta} (t-r)} dw_k(t)\right]^2 \dif r \nonumber\\
&=  \int_0^{T}\sum_{k=1}^N \Bigg[\frac{\si^2}{C_{N,T}^2}
\lambda_k^{4\beta+2\gamma} u_k^2(r)
+  2\frac{\si^3}{C_{N,T}^2} \lambda_k^{4\beta+\gamma}  u_k(r)  \int_r^{T}   e^{-\theta_0 \lambda_k^{2\beta}(t-r)} \dif w_k(t)\\ \nonumber
&\qquad \qquad \qquad \qquad + \frac{\si^4}{C_{N,T}^2} \lambda_k^{4\beta}
\left(\int_r^{T}   e^{-\theta_0 \lambda_k^{2\beta} (t-r)} \dif w_k(t)\right)^2 \Bigg]\dif r \nonumber\\
&=:A_1+ A_2+A_3. 
\end{align*}
It is easy to see that  for any process $\Phi=\{\Phi(s), s\in [0,t]\}$ such that $\sqrt{\Var(\Phi(s))}$ is integrable on $[0,t]$,
it holds that
$\sqrt{ \Var  \left( \int_0^t \Phi_s \dif s \right)} \le \int_0^t  \sqrt{\Var (\Phi_s)} \dif s.$
Therefore, we have
\begin{align*}
\sqrt{ \Var \left(\frac{1}{2} \| C_{N,T} D\wh{F}_{N,T}\|_{\mathcal{H}}^2 \right) } &\le \frac{\sqrt{3}}{2}\left(  \sqrt{ \Var(A_1) } +  \sqrt{ \Var(A_2)  } +  \sqrt{\Var(A_3)  } \right) \\
&\leq  \frac{\sqrt{3}}{2}\left(B_1+B_2+B_3\right),
\end{align*}
where
\begin{align*}
B_1&:=\frac{\si^2}{C_{N,T}^2}\int_0^{T} \left[\Var\left( \sum_{k=1}^N \lambda_k^{4\beta+2\gamma} u_k^2(r)\right) \right]^{1/2} \dif r\\
B_2&:=2\frac{\si^3}{C_{N,T}^2}\int_0^{T} \left[\Var \left( \sum_{k=1}^N \lambda_k^{4\beta+\gamma}  u_k(r) \int_r^{T}   e^{-\theta_0 \lambda_k^{2\beta}(t-r)}
dw_k(t)\right) \right]^{1/2} \dif r\\
B_3&:= \frac{\si^4}{C_{N,T}^2} \int_0^{T} \left[\Var \left( \sum_{k=1}^N \lambda_k^{4\beta}   \left(\int_r^{T}   e^{-\theta_0 \lambda_k^{2\beta}(t-r)}
dw_k(t)\right)^2\right) \right]^{1/2} \dif r.
\end{align*}
Note that by \eqref{eq:2ndmomentuk} and \eqref{eq:4thmomentuk},
\begin{align}\label{var-uk2}
\Var\left(u_k^2(r)\right)  &=\E\left(u_k^4(r)\right)-\E^2\left(u_k^2(r)\right)
=\frac{\si^4 \lambda_k^{-4\beta-4\ga}}{2\theta_0^2} \left(1-e^{-2\theta_0 \lambda_k^{2\beta} r}\right)^2.
\end{align}
For $B_1$, by the independence of $\{u_k\}_{k\ge 1}$ and \eqref{var-uk2}, 
 \begin{align}
B_1&=\frac{\si^2}{C_{N,T}^2}\int_0^{T} \left(\sum_{k=1}^N \lambda_k^{8\beta+4\gamma} \Var\left(u_k^2(r)\right) \right)^{1/2}\dif r\nonumber\\
&=\frac{\si^4}{\sqrt{2}\theta_0C_{N,T}^2}\int_0^{T} \left(\sum_{k=1}^N \lambda_k^{4\beta}
\left(1-e^{-2\theta_0 \lambda_k^{2\beta} r}\right)^2 \right)^{1/2} \dif r \nonumber\\
&\simeq \frac{\si^4T}{\sqrt{2}\theta_0C_{N,T}^2}\left(\sum_{k=1}^N\lambda_k^{4\beta}\right)^{1/2}\quad \mbox{as}\ T\to \infty \nonumber\\
&\sim \frac{1}{N^{1/2}}\to 0, \quad \mbox{as}\ N,T \to \infty.\label{A1_conv}
\end{align}
For $B_2$, we note that $u_k$ and $W_l$ are independent if $k\neq l$.
Therefore, we rewrite $B_2$ as
\begin{align*}
B_2&=\frac{2\si^3}{C_{N,T}^2}\int_0^{T} \left[ \sum_{k=1}^n \lambda_k^{8\beta+2\gamma} \Var \left(u_k(r) \int_r^{T}   e^{-\theta_0 \lambda_k^{2\beta}(t-r)} \dif W_k(t) \right)\right]^{1/2} \dif r.
\end{align*}
By straightforward calculations, we have that
\begin{align*}
\Var \left(u_k(r) \int_r^{T}
e^{-\theta_0 \lambda_k^{2\beta} (t-r)}  \dif w_k(t) \right)
&\le \E\left[ u_k^2(r)\left[   \int_r^{T}   e^{-\theta_0 \lambda_k^{2\beta} (t-r)} \dif w_k(t)\right)^2 \right]\nonumber\\
& = \E\left[ u^2_k(r)\E\left[ \left(  \int_r^{T}   e^{-\theta_0 \lambda_k^{2\beta} (t-r)}  \dif w_k(t)\right)^2\Big| \cf_r \right] \right]\nonumber\\
&= \frac{\si^2 \lambda_k^{-2\beta-2\ga} }{2\theta_0} \left(1-e^{-2\theta_0 \lambda_k^{2\beta}r}\right)  \int_r^{T}   e^{-2\theta_0 \lambda_k^{2\beta}(t-r)} \dif t\nonumber\\
&\le \frac{\si^2\lambda_k^{-4\beta-2\ga} }{4\theta^2_0}.
\end{align*}
Therefore, we get
\begin{equation}
B_2\le \frac{\si^4}{\theta_0 C_{N,T}^2}\int_0^{T} \left( \sum_{k=1}^N \lambda_k^{4\beta} \right)^{1/2} \dif r \sim \frac{1}{N^{1/2}}\to 0, \quad \mbox{as }
N,T\to \infty.
\label{A2_conv}
\end{equation}
Let us now consider $B_3$. Since $w_k$ and $w_j$ are independents for $k\ne j$, we have
 \begin{align}
B_3&= \frac{\si^4}{C_{N,T}^2} \int_0^{T} \left[\Var \left( \sum_{k=1}^N \lambda_k^{4\beta}   \left(\int_r^{T}   e^{-\theta_0 \lambda_k^{2\beta}(t-r)}
\dif w_k(t)\right)^2\right) \right]^{1/2} \dif r\nonumber\\
&=   \frac{\si^4}{C_{N,T}^2} \int_0^{T} \left[ \sum_{k=1}^N \lambda_k^{8\beta} \Var \left(\int_r^{T}   e^{-\theta_0 \lambda_k^{2\beta}(t-r)}
\dif w_k(t)\right)^2 \right]^{1/2} \dif r\nonumber\\
&\le  \frac{\si^4}{C_{N,T}^2} \int_0^{T} \left[ \sum_{k=1}^N \lambda_k^{8\beta}  \E \left(\int_r^{T}   e^{-\theta_0 \lambda_k^{2\beta}(t-r)}
\dif w_k(t)\right)^4 \right]^{1/2} \dif r\nonumber\\
&= \frac{\si^4}{C_{N,T}^2} \int_0^{T} \left[3 \sum_{k=1}^N \lambda_k^{8\beta}  \left[ \E \left(\int_r^{T}   e^{-\theta_0 \lambda_k^{2\beta}(t-r)}
\dif w_k(t)\right)^2\right]^2 \right]^{1/2} \dif r\nonumber\\
&\leq \frac{\sqrt{3}\si^4}{2\theta_0 C_{N,T}^2} T \left( \sum_{k=1}^N \lambda_k^{4\beta}  \right)^{1/2}\sim \frac{1}{N^{1/2}} \rightarrow 0, \mbox{ as } N,T \rightarrow \infty. \label{A3_conv}
\end{align}
Finally, combining \eqref{A1_conv}, \eqref{A2_conv} and \eqref{A3_conv}, we have that for every $\varepsilon>0$,
there exist two independent constants $N_0,T_0>0$ such that for all $N\geq N_0$ and $T\geq T_0$,
$$
B_1+B_2+B_3  < \varepsilon.
$$
This completes the proof. \end{proof}

%%%%%%%%%%%%%%%%%%%%%%%%%%%%

\begin{proof}[Proof of Lemma~\ref{lemma:Discr1}] Using \eqref{eq:u_k}, \eqref{eq:productuk} follows by direct evaluations.
As far as \eqref{eq:continuityuk} and \eqref{eq:bddnessuk}, since $u_k(t)-u_k(s)$ is a Gaussian random variable, it is enough to prove \eqref{eq:continuityuk} and \eqref{eq:bddnessuk} for $l=1$. We note that
for $t<s$,
\begin{align*}
\mathbb{E}|u_k(t)-u_k(s)|^2&=\sigma^2\lambda_k^{-2\gamma}
\mathbb{E}\left[\int_0^t e^{-\theta_0 \lambda_k^{2\beta}(t-r)} \dif w_k(r)-\int_0^s e^{-\theta_0 \lambda_k^{2\beta}(s-r)}\dif w_k(r)\right]^2\\
&=\sigma^2\lambda_k^{-2\gamma}\left[
\int_0^t e^{-2\theta_0 \lambda_k^{2\beta}(t-r)}\dif r - 2\int_0^t e^{-\theta_0 \lambda_k^{2\beta} (t+s-2r)}\dif r
+\int_0^s e^{-2\theta_0 \lambda_k^{2\beta}(s-r)}\dif r\right]\\
&=\frac{\sigma^2\lambda_k^{-2\gamma-2\beta}}{2\theta_0}
\left[2(1-e^{-\theta_0 \lambda_k^{2\beta}(s-t)}) + (e^{-\theta_0 \lambda_k^{2\beta} t}+e^{-\theta_0 \lambda_k^{2\beta}s})
(e^{-\theta_0 \lambda_k^{2\beta} s}-e^{-\theta_0 \lambda_k^{2\beta} t})\right]\\
&\leq C\sigma^2\lambda_k^{-2\gamma}|t-s|,
\end{align*}
for some $C>0$, and where in the last inequality we used the fact that $e^{-x}$ is Lipschitz continuous on $[0,\infty)$. Thus part \eqref{eq:continuityuk} is proved.
Finally, we have the estimates
\begin{align*}
\mathbb{E}|u_k(t)+u_k(s)|^2&=\sigma^2\lambda_k^{-2\gamma}
\mathbb{E}\left[\int_0^t e^{-\theta_0 \lambda_k^{2\beta}(t-r)} \dif w_k(r)+\int_0^s e^{-\theta_0 \lambda_k^{2\beta}(s-r)}\dif w_k(r)\right]^2\\
&=\sigma^2\lambda_k^{-2\gamma}\left[
\int_0^t e^{-2\theta_0 \lambda_k^{2\beta}(t-r)}\dif r +2\int_0^t e^{-\theta_0 \lambda_k^{2\beta} (t+s-2r)}\dif r
+\int_0^s e^{-2\theta_0 \lambda_k^{2\beta}(s-r)}\dif r\right]\\
&\leq C\sigma^2\lambda_k^{-2\gamma-2\beta}T,
\end{align*}
for some $C>0$, which implies \eqref{eq:bddnessuk}. The proof is complete.
\end{proof}

\begin{proof}[Proof of Lemma~\ref{lemma:Y-Difference}]
Since $u_k, \ k \geq 1$, are independent, and taking into account that
$$
\int_0^T u_k(t)\dif w_k(t)=\sum_{i=1}^M \int_{t_{i-1}}^{t_i}u_k(t)\dif w_k(t),
$$
we have that
\begin{align*}
\mathbb{E}|Y_{N,M,T}-Y_{N,T}|^2 &
=\mathbb{E}\left|\sum_{k=1}^N \lambda_k^{2\beta+\gamma} \left[\sum_{i=1}^M u_k(t_{i-1})\left(
w_k(t_i)-w_k(t_{i-1})\right)-\int_0^T u_k(t)\dif w_k(t)\right] \right|^2\\
&=\sum_{k=1}^N \lambda_k^{4\beta+2\gamma}\mathbb{E}
\left|\sum_{i=1}^M u_k(t_{i-1})\left( w_k(t_i)-w_k(t_{i-1})\right)-\int_0^T u_k(t)\dif w_k(t)\right|^2 \\
&=\sum_{k=1}^N \lambda_k^{4\beta+2\gamma} \mathbb{E}
\left|\sum_{i=1}^M \int_{t_{i-1}}^{t_i} \left(u_k(t_{i-1})-u_k(t)\right)\dif w_k(t) \right|^2\\
&=\sum_{k=1}^N \lambda_k^{4\beta+2\gamma}\sum_{i=1}^M \int_{t_{i-1}}^{t_i} \mathbb{E}\left(u_k(t_{i-1})-u_k(t)\right)^2\dif t
\end{align*}
and hence, by \eqref{eq:continuityuk}, there exist constants $C_1,C_2>0$, such that
\begin{align*}
\mathbb{E}|Y_{N,M,T}-Y_{N,T}|^2
&\leq C_1\sum_{k=1}^N \lambda_k^{4\beta}\sum_{i=1}^M \int_{t_{i-1}}^{t_i} |t_{i-1}-t|\dif t\\
&=C_1\left(\sum_{k=1}^N \lambda_k^{4\beta}\right)\frac{T^2}{2M}
\leq C_2\frac{T^2 N^{\frac{4\beta}{d}+1}}{M}.
\end{align*}
Hence, \eqref{eq:Y-NMT} follows at once.

Next we will prove \eqref{eq:I-NMT}. We note that
\begin{align*}
\mathbb{E}|I_{N,M,T}-I_{N,T}|^2&=
\mathbb{E}\left|\sum_{k=1}^N \lambda_k^{4\beta+2\gamma}\left(\sum_{i=1}^M u_k^2(t_{i-1})(t_i-t_{i-1})
-\int_0^T u_k^2(t)\dif t\right)\right|^2\\
&=\sum_{k=1}^N \lambda_k^{8\beta+4\gamma} \mathbb{E}\left|
\sum_{i=1}^M \int_{t_{i-1}}^{t_i} \left(u_k^2(t_{i-1})-u_k^2(t)\right)\dif t \right|^2.
\end{align*}
Consequently, letting $U_i(t):=u_k^2(t_{i-1})-u_k^2(t), \ k\geq 1$, we continue
\begin{align*}
\mathbb{E}|I_{N,M,T}-I_{N,T}|^2
&=\sum_{k=1}^N \lambda_k^{8\beta+4\gamma} \sum_{i=1}^M \mathbb{E}\left|
 \int_{t_{i-1}}^{t_i} U_i(t)\dif t
\right|^2+2\sum_{k=1}^N \lambda_k^{8\beta+4\gamma} \sum_{i<j}\mathbb{E}
 \int_{t_{i-1}}^{t_i}\int_{t_{j-1}}^{t_j} U_i(t)U_j(s) \dif s\dif t\\
&=: I_1+I_2.
\end{align*}
Note that by Cauchy-Schwartz inequality,
\begin{align*}
\mathbb{E}|U_i^2(t)|&=\mathbb{E}|u_k^2(t_{i-1})-u_k^2(t)|^2
=\mathbb{E}|u_k(t_{i-1})-u_k(t)|^2|u_k(t_{i-1})+u_k(t)|^2\\
&\leq \left(\mathbb{E}|u_k(t_{i-1})-u_k(t)|^4\right)^{1/2}
\left(\mathbb{E}|u_k(t_{i-1})+u_k(t)|^4\right)^{1/2}.
\end{align*}
Moreover, by \eqref{eq:continuityuk} and \eqref{eq:bddnessuk},
\begin{equation}
\label{Usquare}
\mathbb{E}|U_i^2(t)|\leq c_1\lambda_k^{-4\gamma-2\beta}T|t-t_{i-1}|,\quad \mbox{for some}\ c_1>0.
\end{equation}
Again by Cauchy-Schwartz inequality and \eqref{Usquare}, we have that
\begin{align*}
I_1&=\sum_{k=1}^N \lambda_k^{8\beta+4\gamma} \sum_{i=1}^M \mathbb{E}\left|  \int_{t_{i-1}}^{t_i} U_i(t)\dif t \right|^2
\leq \sum_{k=1}^N \lambda_k^{8\beta+4\gamma} \sum_{i=1}^M  (t_i-t_{i-1})
\int_{t_{i-1}}^{t_i} \mathbb{E}|U_i^2(t)|\dif t\\
&\leq c_1T\sum_{k=1}^N \lambda_k^{6\beta} \sum_{i=1}^M  (t_i-t_{i-1})
\int_{t_{i-1}}^{t_i} (t-t_{i-1})\dif t
=c_1\sum_{k=1}^N \lambda_k^{6\beta} \frac{T^4}{2M^2}.
\end{align*}
Turning to $I_2$, we first notice that
\begin{align*}
\mathbb{E}|U_i(t)U_j(s)|&=\mathbb{E}\Big[\left(u_k^2(t_{i-1})-u_k^2(t) \right)
\left(u_k^2(t_{j-1})-u_k^2(s) \right)\Big]\\
&=\mathbb{E}\Big[\left(u_k(t_{i-1})-u_k(t)\right)
\left(u_k(t_{i-1})+u_k(t)\right)\left(u_k(t_{j-1})-u_k(s)\right)\left(u_k(t_{j-1})+u_k(s)\right)\Big]\\
&=\mathbb{E}\Big[\left(u_k(t_{i-1})-u_k(t)\right)\left(u_k(t_{j-1})-u_k(s)\right)u_k(t_{i-1})u_k(t_{j-1})\Big]\\
&+\mathbb{E}\Big[\left(u_k(t_{i-1})-u_k(t)\right)\left(u_k(t_{j-1})-u_k(s)\right)(u_k(t_{i-1})u_k(s)\Big]\\
&+\mathbb{E}\Big[\left(u_k(t_{i-1})-u_k(t)\right)\left(u_k(t_{j-1})-u_k(s)\right)u_k(t)u_k(t_{j-1})\Big]\\
&+\mathbb{E}\Big[\left(u_k(t_{i-1})-u_k(t)\right)\left(u_k(t_{j-1})-u_k(s)\right)u_k(t)u_k(s)\Big].
\end{align*}
By the Wick's Lemma \cite[Lemma 3.1]{Bishwal2008}, we continue
\begin{align*}
\mathbb{E}|U_i(t)U_j(s)|
&=\mathbb{E}\left[\left(u_k(t_{i-1})-u_k(t)\right)
\left(u_k(t_{i-1})+u_k(t)\right)\right]\mathbb{E}\left[\left(u_k(t_{j-1})-u_k(s)\right)\left(u_k(t_{j-1})+u_k(s)\right)\right]\\
&+\mathbb{E}\left[\left(u_k(t_{i-1})-u_k(t)\right)\left(u_k(t_{j-1})-u_k(s)\right)
\right]\mathbb{E}\left[\left(u_k(t_{i-1})+u_k(t)\right)\left(u_k(t_{j-1})+u_k(s)\right)\right]\\
&+\mathbb{E}\left[\left(u_k(t_{i-1})-u_k(t)\right)\left(u_k(t_{j-1})+u_k(s)\right)
\right]\mathbb{E}\left[\left(u_k(t_{i-1})+u_k(t)\right)\left(u_k(t_{j-1})-u_k(s)\right)\right]\\
&=:J_1+J_2+J_3.
\end{align*}
For $J_2$, we have
\begin{align*}
\mathbb{E}\left(u_k(t_{i-1})-u_k(t)\right)\left(u_k(t_{j-1})-u_k(s)\right)&=
\mathbb{E}\left(u_k(t_{i-1})u_k(t_{j-1})\right)
-\mathbb{E}\left(u_k(t_{i-1})u_k(s)\right)\\
&-\mathbb{E}\left(u_k(t)u_k(t_{j-1})\right)
+\mathbb{E}\left(u_k(t)u_k(s)\right).
\end{align*}
By \eqref{eq:productuk}, for $i<j$ and $t<s$,
\begin{align*}
&\mathbb{E}\left(u_k(t_{i-1})-u_k(t)\right)\left(u_k(t_{j-1})-u_k(s)\right)=
\frac{\sigma^2\lambda_k^{-2\gamma-2\beta}}{2\theta_0}
\Bigg[e^{-\theta_0 \lambda_k^{2\beta}(t_{j-1}-t_{i-1})}-e^{-\theta_0 \lambda_k^{2\beta}(t_{j-1}+t_{i-1})}\nonumber \\
& \ -e^{-\theta_0 \lambda_k^{2\beta}(s-t_{i-1})}+e^{-\theta_0 \lambda_k^{2\beta}(t_{i-1}+s)} -e^{-\theta_0 \lambda_k^{2\beta}(t_{j-1}-t)}+e^{-\theta_0 \lambda_k^{2\beta}(t_{j-1}+t)} +e^{-\theta_0 \lambda_k^{2\beta}(s-t)}-e^{-\theta_0 \lambda_k^{2\beta}(s+t)} \Bigg] \nonumber\\
&= \frac{\sigma^2\lambda_k^{-2\gamma-2\beta}}{2\theta_0}
\Bigg[\left(e^{-\theta_0 \lambda_k^{2\beta}(t_{j-1}-t_{i-1})} -e^{-\theta_0 \lambda_k^{2\beta}(t_{j-1}-t)}\right)
+\left(e^{-\theta_0 \lambda_k^{2\beta}(t_{j-1}+t)} -e^{-\theta_0 \lambda_k^{2\beta}(t_{j-1}+t_{i-1})}\right) \nonumber\\
& \ +\left(e^{-\theta_0 \lambda_k^{2\beta}(s-t)} -e^{-\theta_0 \lambda_k^{2\beta}(s-t_{i-1})}\right)+
\left(e^{-\theta_0 \lambda_k^{2\beta}(t_{i-1}+s)} -e^{-\theta_0 \lambda_k^{2\beta}(s+t)}\right) \Bigg] \nonumber \\
&\leq c_2 \lambda_k^{-2\gamma}(t-t_{i-1}),
\end{align*}
for some $c_2>0$. By similar arguments, we also obtain
$$
\mathbb{E}\left(u_k(t_{i-1})+u_k(t)\right)\left(u_k(t_{j-1})+u_k(s)\right)\leq
c_3\lambda_k^{-4\gamma}(s-t_{j-1}),
$$
for some $c_3>0$. Thus,
$$
J_2 \leq c_4\lambda_k^{-4\gamma}(t-t_{i-1})(s-t_{j-1}),
$$
for some $c_4>0$. By analogy, one can treat $J_1$ and $J_3$, and derive the following upper bounds:
$$
J_1\leq c_5\lambda_k^{-4\gamma}(t-t_{i-1})(s-t_{j-1}), \qquad
J_3 \leq c_6\lambda_k^{-4\gamma}(t-t_{i-1})(s-t_{j-1}),
$$
for some $c_5,c_6>0$. Finally, combining the above, we have
\begin{align*}
I_2&\leq c_7\sum_{k=1}^N \lambda_k^{8\beta} \sum_{i<j} \int_{t_{j-1}}^{t_j}\int_{t_{i-1}}^{t_i}(t-t_{i-1})(s-t_{j-1})\dif t \dif s \\
&\leq c_8\sum_{k=1}^N \lambda_k^{8\beta} \frac{T^4}{M^2},
\quad \mbox{for some}\ c_8,c_9>0.
\end{align*}
Thus, using the estimates for $I_1,I_2$, and the fact that $\lambda_k\sim k^{1/d}$, we conclude that
$$
I_1+I_2 \leq  c_9\sum_{k=1}^N \lambda_k^{8\beta} \frac{T^4}{M^2} \leq C_2 \frac{T^4 N^{\frac{8\beta}{d}+1}}{M^2},
$$
and hence \eqref{eq:I-NMT} is proved.
The estimate \eqref{eq:V-NMT} is proved by similar arguments, and we omit the details here.
This completes the proof.
\end{proof}

\subsection{Auxiliary results}
For reader's convenience, we present here some simple, or well-known, results from probability.
Let  $X,Y,Z$ be random variables, and assume that $Z>0$ a.s.. For any $\varepsilon>0$ and $\delta\in(0,\varepsilon/2)$, the following inequalities hold true.
\begin{align}
& \bP(|Y/Z|>\varepsilon) \leq \bP(|Y|>\delta) + \bP(|Z-1|> (\varepsilon-\delta)/\varepsilon), \label{eq:simpleOne}\\
&  \sup_{x\in\bR} \Big|  \bP(X+Y\leq x) - \Phi(x)\Big|   \leq \sup_{x\in\bR} \Big|  \bP(X+Y\leq x) - \Phi(x)\Big| + \bP(|Y|>\varepsilon) + \varepsilon, \label{eq:simple2} \\
&  \sup_{x\in\bR} \Big|  \bP(Y/Z\leq x) - \Phi(x)\Big|   \leq \sup_{x\in\bR} \Big|  \bP(Y\leq x) - \Phi(x)\Big| + \bP(|Z-1|>\varepsilon) + \varepsilon,  \label{eq:simple3}
\end{align}
where $\Phi$ denotes the probability function of a standard Gaussian random variable.

\section{Elements of Malliavin calculus}\label{apendix:MalCal}
In this section, we recall some facts from Malliavin calculus associated with a Gaussian process, that we use in Section~\ref{sec:MalCal}. For more details, we refer to~\cite{Nualart2006}. Toward this end, let $T>0$ be given. We consider the space $\mathcal{H}=L^2\left([0,T]\times \mathcal{M}\right)$, where $\mathcal{M}$ is the counting measure on $\bN$, namely, for $v\in\ch$,
$$
v(t)=\sum_{k=1}^\infty v_k(t).
$$
We endow $\cH$ with the inner product and the norm
\begin{equation}\label{innerprod-1}
\langle u,v\rangle_{\mathcal{H}}:= \sum_{k=1}^\infty\int_0^T u_k(t) v_k(t) \dif t,\quad
\mbox{and}\quad
\|v\|_{\mathcal{H}}:= \sqrt{\langle v,v\rangle_{\mathcal{H}}}, \quad \ u,v\in \ch.
\end{equation}
We fix an isonormal Gaussian process $W=\{W(h)\}_{h\in\mathcal{H}}$ on $\mathcal{H}$, defined on a suitable probability space $(\Omega, \sF,\mathbb{P})$, such that $\sF=\sigma(W)$ is the $\sigma$-algebra generated by $W$.
Denote by $C_p^{\infty}(\RR^n)$, the space of all smooth functions on $\bR^n$ with at most polynomial growth partial derivatives.
Let $\mathcal{S}$ be the space of simple functionals of the form
\begin{equation*}
F = f(W(h_1), \dots, W(h_n)),\quad
f\in C_p^{\infty}(\RR^n), \  h_i \in \ch,\ 1\leq i \leq n.
\end{equation*}
As usual, we define the Malliavin derivative $D$ on $\mathcal{S}$ by
\begin{equation}\label{def-MD}
DF=\sum_{i=1}^n  \frac {\partial f} {\partial x_i} (W(h_1), \dots, W(h_n)) h_i, \quad F\in \mathcal{S}.
\end{equation}
We note that the derivative operator $D$ is a closable operator from $L^p(\Omega)$ into $L^p(\Omega;  \ch)$,  for any $p \geq1$.
Let $\mathbb{D}^{1,p}$, $p \ge 1$, be the completion of $\mathcal{S}$
with respect to the norm
$$
\|F\|_{1,p} = \left(\E\big[ |F|^p \big] + \E\big[  \|D F\|^p_\ch \big] \right)^{1/p}.
$$
Also, for $F$ of the form
$$
F=f\left(W\left(\1_{[0,t_1]}\right),\dots,W\left(\1_{[0,t_n]}\right)\right),
\quad t_1,\dots,t_n\in [0,T],
$$
we define the Malliavin derivative of $F$ at the point $t$ as
$$
D_t F=\sum_{i=1}^n \frac{\partial f}{\partial x_i}\left(W\left(\1_{[0,t_1]}\right),\dots,W\left(\1_{[0,t_n]}\right)\right)
\1_{[0,t_i]}(t),\quad t\in [0,T],
$$
where $\1_{A}$ denotes the indicator function of set $A$. For simplicity,
from now on, we define $W(t):=W\left(\1_{[0,t]}\right)$, $t\in [0,T]$, to represent a standard Brownian motion on $[0,T]$.
If $\sF$ is generated by  a collection of independent standard Brownian motions $\{W_k,\ k\geq 1\}$ on $[0,T]$, we define the Malliavin derivative of $F$ at the point $t$ by
\begin{equation}\label{def-MD-tj}
D_tF:=\sum_{k=1}^{\infty}D_{t,k}F:=\sum_{k=1}^{\infty}\sum_{i=1}^n \frac{\partial f}{\partial x_i}\left(W_k(t_1),\dots,W_k(t_n)\right)
\1_{[0,t_i]}(t),\quad t\in [0,T].
\end{equation}
Next, we denote by $\boldsymbol{\delta}$, the adjoint of the Malliavin derivative $D$ (as defined in
\eqref{def-MD})
given by the duality formula
\begin{equation*}
\E\left(\boldsymbol{\delta}(v) F\right) = \E\left(\langle v, DF \rangle_\ch\right),
\end{equation*}
for $F \in \mathbb{D}^{1,2}$ and $v\in \mathcal{D}(\boldsymbol{\delta})$, where $\mathcal{D}(\boldsymbol{\delta})$ is the domain of $\boldsymbol{\delta}$. If $v\in  L^2(\Omega;\ch)\cap \mathcal{D}(\boldsymbol{\delta})$ is a square integrable process, then the adjoint $\boldsymbol{\delta}(v)$ is called the Skorokhod integral of the process $v$
(cf. \cite{Nualart2006}), and it can be written as
\begin{align*}
 \boldsymbol{\delta}(v)=\int_0^T v(t) \dif W(t).
\end{align*}

\begin{proposition}\cite[Theorem 1.3.8]{Nualart2006} \label{chain-rule-Malliavin}
Suppose that $v \in L^2(\Omega;\ch)$ is a square integrable process such that $v(t)\in \mathbb{D}^{1,2}$ for almost all $t\in [0,T]$.
Assume that the two parameter process $\{D_tv(s)\}$ is square integrable in $L^2\left([0,T]\times \Omega;\ch\right)$.
Then, $\boldsymbol{\delta}(v) \in \mathbb{D}^{1,2}$ and
\begin{align}\label{chain-rule-formula}
D_t \left(\boldsymbol{\delta}(v)\right) = v(t) + \int_0^{T} D_tv(s) \dif W(s),\quad t\in [0,T].
\end{align}
\end{proposition}
Next, we present a connection between Malliavin calculus and Stein's method. For symmetric functions $f\in  L^2\big([0,T]^q\big)$, $q\geq 1$, let us define the following multiple integral of order $q$ 
\begin{equation*}
 \bI_q(f)=q!\int_0^T \dif W(t_1)\int_0^{t_1} \dif W(t_2)\cdots \int_0^{t_{q-1}} \dif W(t_q) f(t_1,\ldots,t_q),
\end{equation*}
with $0<t_1<t_2<\cdots<t_q<T$. Note that $\bI_q(f)$ is also called the $q$-th Wiener chaos \cite[Theorem 2.7.7]{NourdinPeccati2012}.
Denote by $d_{TV}(F,G)$, the total variation of two random variables $F$ and $G$.
\begin{theorem} \cite[Corollary 5.2.8]{NourdinPeccati2012} \label{Equiv}
Let $F_N=\bI_q(f_N)$, $N\ge 1$, be a sequence of random variables
for some fixed integer $q\ge 2$.
Assume that $\E\left(F_N^2\right)\rightarrow \sigma^2>0$, as $N\rightarrow\infty$.
Then, as $N \rightarrow\infty$, the following assertions are equivalent:
\begin{enumerate}
\item $F_N\overset{d}\longrightarrow \cn:=\cN(0,\sigma^2)$;
\item $d_{TV}\left(F_N,\cn\right)\longrightarrow 0$.
\end{enumerate}
\end{theorem}
We conclude this section with a result about an upper bound for the total variation of a $q$-th multiple integral and a Gaussian random variable.
\begin{proposition} \cite[Theorem 5.2.6]{NourdinPeccati2012} \label{Th2.2}
Let $q\ge 2$ be an integer, and let $F=\bI_q(f)$ be a multiple integral of order $q$ such that $\E(F^2)=\sigma^2>0$. Then, for $\cn=\cN(0,\sigma^2)$,
\begin{equation*}
d_{TV}(F,\cn)\le \frac{2}{\sigma^2}\sqrt{\Var\left(\frac{1}{q}\|DF \|_{\mathcal{H}}^2 \right)}.
\end{equation*}
\end{proposition}

\bibliographystyle{alpha}
\def\cprime{$'$}

\end{document}